\documentclass[11pt,leqno]{amsart}
\usepackage{amsmath, amsthm, amssymb, xspace, mathrsfs}
\usepackage[breaklinks=true]{hyperref}
\usepackage{amsmath,amscd}
\newtheorem{assumption}{Assumption}
\newtheorem*{acknowledgement}{Acknowledgements}

\theoremstyle{plain}
\newtheorem{theorem}[equation]{Theorem}
\newtheorem{lemma}[equation]{Lemma}
\newtheorem{prop}[equation]{Proposition}
\newtheorem{cor}[equation]{Corollary}

\theoremstyle{definition}
\newtheorem{defin}[equation]{Definition}

\numberwithin{equation}{section}

\newcommand{\comment}[1]{}

\begin{document}

\title[Commutants of multiplication operators]{Multiplication operators on the Bergman spaces of Pseudoconvex domains  }

\author{Akaki Tikaradze}
\address{Department of Mathematics, University of Toledo, Toledo, Ohio 43606}
\email{\tt tikar06@gmail.com}
\begin{abstract}

Let $\Omega\subset \mathbb{C}^n$ be a bounded smooth pseudoconvex domain, 
and let $f=(f_1,\cdots, f_n):\overline{\Omega}\subset\mathbb{C}^n$ be
an $n$-tuple of holomorphic functions on $\overline{\Omega}.$ In this paper we study  
commutants of the corresponding multiplication operators $\lbrace T_{f_1}, \cdots, T_{f_n}\rbrace=T_f $ on the 
Bergman space $A^2(\Omega).$
One of our main results is a geometric description of the algebra of commutants of $\lbrace T_f, {T_f}^{*}\rbrace, $
generalizing a result by Douglas, Sun and Zheng \cite{DSZ}.

\end{abstract}
\maketitle

\section{Introduction}

Let $\Omega\subset \mathbb{C}^n$ be a bounded smooth pseudoconvex  domain. The Bergman
space of all square integrable holomorphic functions on $\Omega$
will be denoted by $A^2(\Omega),$ while the subspace
of all bounded holomorphic functions on $\Omega$ will be 
denoted by $H^{\infty}(\Omega).$
 Given a function 
$f\in L^{\infty}(\Omega),$ one defines the corresponding Toeplitz
operator with the symbol $f: T_f:A^2(\Omega)\to A^2(\Omega),$ 
as the composition of the multiplication operator by $f$  followed
by the orthogonal projection
from $L^2(\Omega)$ to $A^2(\Omega).$
If $f$ is holomorphic, then $T_f=M_f$ is the multiplication operator by $f.$
Questions related to commutants of Toeplitz
operators have been of great interest for some time.

The following is the motivating problem for this paper: Let
\\$f=(f_1,\cdots, f_n):\overline{\Omega}\to \mathbb{C}^n$ be a holomorphic
mapping in a neighbourhood of $\overline{\Omega}$ with a 
nontrivial Jacobian determinant. Describe the algebra of commutants of
 $\lbrace T_{f_i}, 1\leq i\leq n\rbrace =T_f.$

It is of a special interests to describe the largest $C^{*}$-subalgebra
of the above algebra, the algebra of  commutants of $\lbrace T_f, T_f^{*}\rbrace$
(here and everywhere $T_f^{*}$ denotes $\lbrace T_{f_i}^{*},1\leq i\leq n\rbrace).$
Indeed, reducing subspaces of $T_f$ correspond to projections 
in  this algebra.

Both of the above questions have been extensively studied for the past several decades when $n=1$ and $\Omega=D$
is the unit disc. Indeed, by a result of Thompson \cite{Th}, it suffices to study
the commutants of $T_f$ when $f$ is a finite Blaschke product. In this case it
can be described it terms of the Riemann surface $f^{-1}\circ f(D')$, where $D'$ is $D$ with
preimages of the critical values of $f$ removed [\cite{Co}, Theorem 3] 
(although Cowen and Thompson worked in the Hardy space setting, 
their results easily
carry over to the Bergman space). 

In a recent important work by Douglas, Sun and Zheng \cite{DSZ},  
the algebra of commutants of $\lbrace T_f,T_f^{*}\rbrace$ is explicitly described.
 In particular, they show that
its dimension equals to the number of connected components of $f^{-1}\circ f(D')$ (\cite{DSZ}, Theorem 7.6).
Also noteworthy are results of Guo and Huang, who under the assumption that $f:D\to f(D)$ is a covering map, 
described among other things the commutant
of $\lbrace T_f, T_f^{*}\rbrace$ in terms of fundamental group of $f(D)$ [\cite{GH2}, Theorem 1.3].

Motivated by these results, we extend them to high dimensional domains. Namely,
we introduce a certain $n$-dimensional complex manifold $W_f$(Definition \ref{sheep})
$$W_f\subset (\Omega\setminus Z)\times_f(\Omega\setminus Z)=\lbrace (z, w), f(z)=f(w), z, w\in \Omega\setminus Z\rbrace$$
defined as the largest open subset of $(\Omega\setminus Z)\times_f(\Omega\setminus Z)$ such that
the projection $p:W_f\to \Omega\setminus Z$ is a covering map, where $Z$ is the 
preimage of all critical values of $f$
on $\overline{\Omega}.$ Under some mild assumptions on $\Omega, f $ (Assumptions \ref{sheep}, \ref{dense})
we prove that the algebra of  commutants of $\lbrace T_f, {T_f}^{*}\rbrace$ is
isomorphic to the algebra of locally constant functions on $W_f$ under convolution product (Theorem \ref{reducing}).
This is a
generalization of the above mentioned theorem by Douglas, Sun and Zheng \cite{DSZ}. Our proof
closely follows their ideas.

We also investigate the commutants of $T_f$ in the Toeplitz algebra of
$\Omega$, the norm closed subalgebra of $B(A^2(\Omega))$ generated by all Toepltiz
operators $T_h, h\in L^{\infty}(\Omega)$.
Motivated by a result of Axler-Cuckovic-Rao \cite{ACR} on commutants of analytic
Toeplitz operators in one variable, we prove that the commutant of $T_f$ in the Toeplitz
algebra of $\Omega$ consists of multiplication operators by bounded holomorphic functions
on $\Omega,$ Theorem \ref{strong}.

\section{Nullstellensatz for the Bergman space}
Throughout this paper for a holomorphic mapping $g:\Omega\to \mathbb{C}^n, \Omega\subset \mathbb{C}^n$  
by $J_g$ we will denote the determinant of the Jacobian of $g.$

In this section we will recall a (weak) version of Nullstellensatz for the Bergman space of a bounded
pseudoconvex domain in $\mathbb{C}^n$ (Lemma \ref{Nullst}). This result will be crucial for studying commutants 
of $T_f.$
All the results in this section
follow well-known approach of using Koszul  and $\bar{\partial}$-complex
for proving Nullstellensatz type statements on pseudoconvex domains and are essentially  well-known
(see for example \cite{PS}).
We include proofs for a reader's convenience.

As always, let $\Omega\subset \mathbb{C}^n$ be a bounded pseudoconvex domain. We will denote
by $A^{\infty}(\Omega)$ the set of all holomorphic functions on $\Omega$ which
are $C^{\infty}$-smooth on $\overline{\Omega}.$ Let $f=(f_1,\cdots, f_m):\Omega\to \mathbb{C}^m$
be an $m$-tuple of holomorphic functions from $A^{\infty}(\Omega),$ which will also be viewed as a holomorphic
mapping to $\mathbb{C}^m.$
Let us recall the definition of the Koszul double complex of $f$ on $\Omega.$
Define  the $\bar{\partial}$-Koszul double complex $(K, b_f, \bar{\partial})$ on $\Omega$ as follows
$$K=\bigoplus K_{i, j},  K_{i, j}=\Lambda^i(V)\otimes_{\mathbb{C}} C^{\infty}_{0, j}(\overline{\Omega})$$ 
where $V=\oplus_{i=1}^m \mathbb{C}v_i,$ and $C^{\infty}_{0, j}$($\overline{\Omega})$
denotes the space
of all $C^{\infty}$-smooth $(0, j)$-forms on $\overline{\Omega}$
There is a natural product on $K$
defined as follows $$(u\otimes \omega_1)\cdot (v\otimes \omega_2)=(u\wedge v)\otimes (\omega_1\wedge \omega_2).$$
 Differentials of this bicomplex
are $\bar{\partial}:K_{i, j}\to K_{i, j+1}$ and the Koszul differential 
$b_f:K_{i, j}\to K_{i-1, j}$ defined as follows
$$ b_f(\sum_i v_i\otimes \omega_i)=\sum_if_i\omega, $$
$$b_f(x\cdot y)=b_f(x)\cdot y+(-1)^ix\cdot b_f(y), x\in K_{i, j}, \bar{\partial}(u\otimes\omega)=u\otimes \bar{\partial}(\omega)$$
Clearly $\bar{\partial}b_f=b_f\bar{\partial}.$

\begin{lemma}\label{koszul}

Let $\Omega\subset \mathbb{C}^n$ be a bounded pseudoconvex domain. 
Let $f=(f_1,\cdots,f_m)\in A^{\infty}(\Omega)$ be an $m$-tuple
of holomorphic functions and  $(K, b_f, \bar{\partial})$ be the Koszul
double complex of $f$ as above. Let
$U\subset\subset \Omega$ be an open subset such that $f^{-1}(0)\cap \bar{U}=\emptyset$. 
 Let $w\in K_{i,j}$  be such that 
$b_f(w)=\bar{\partial}(w)=0,$ $supp(w)\subset U$ then there exists $w'\in K_{i+1, j}$ 
such that $w=b_fw', \bar{\partial}(w')=0.$

\end{lemma}
\begin{proof}

Let $w\in K_{i,j}.$ We will proceed by the descending induction on $i.$
 There exists $y\in K_{i+1, j}$ such that $b_fy=w, supp(y)\subset U.$ 
Indeed, let $g_i\in C^{\infty}(\overline{\Omega})$ be such that $(\sum_i f_ig_i)_U=1.$ 
Therefore  $b_f((\sum v_i\otimes g_i)\cdot w)=w.$
 Then $\bar{\partial}(y)\in K_{i+1, j+1}$ satisfies the
inductive assumption, so there exists $z$ such that $b_f(z)=\bar{\partial}(y)$ and $\bar{\partial}(z)=0.$
Let $z_1$ be such that $\bar{\partial}(z_1)=z$ (it exists by Kohn's theorem). Replacing $y$ by $y-b_f(z_1)$ we are done.

\end{proof}

\begin{cor}

Let $f_1, \cdots, f_n\in A^{\infty}(\Omega)$ and let $U\subset\subset \Omega$ be a 
an open subset of $\Omega$ such that $f^{-1}(0)\subset U.$ If $g\in A^{\infty}(\Omega)$
such that $g\in \sum_i f_iA(U),$ then $g\in \sum_if_iA^{\infty}(\Omega).$

\end{cor}

\begin{proof}

Let $h_i\in C^{\infty}(\bar{\Omega})\cap A(U)$ such that $g=\sum_if_ih_i.$ Then
$bx=\bar{\partial}(x)=0$ where $x=\sum v_i\otimes \bar{\partial}(h_i).$ Thus
by the above there exists $z\in K_{2, 0}$ such that $x=b(\bar{\partial}(z)).$
Then $\bar{\partial}(\sum v_i\otimes h_i-b(z))=0$ and $b((\sum v_i\otimes h_i-b_f(z))=g.$
Write $\sum f_i\otimes h_i-b(z)=\sum v_i\otimes h_i.$ Then 
$h_i\in A^{\infty}(\overline{\Omega})$ and $g=\sum_i f_ih_i.$
\end{proof}

For a subset $B\subset \overline{\Omega}$, we will denote by $I(B)$ the ideal of holomorphic functions
on $\Omega$ which vanish on $B.$

The proof below directly follows the proofs of similar statements by Overlid \cite{Ov},
Hakim-Sibony \cite{HS}.

\begin{cor}\label{zero} Let $f=f_1,\cdots, f_m\in A^{\infty}(\Omega)$ be such that
$f^{-1}(0)$ is a finite set. If the Jacobian of $f$
has the full rank on each point of $f^{-1}(0),$ then 
$I(f^{-1}(0))\cap A^{\infty}(\overline{\Omega})=\sum_i f_iA^{\infty}(\overline{\Omega}).$

\end{cor}

\begin{proof}

Let $h\in I(F^{-1}(0))\cap A^{\infty}(\overline{\Omega}).$
It follows from the local Nullstellensatz that there exists an open neighbourhood of $f^{-1}(0),$
$f^{-1}(0)\subset U\subset \Omega$ and $g_i\in A(U),$ such that $h|_{U}=\sum_if_i|_{U}g_i.$
By the above corollary we are done.

\end{proof}

We will need the following assumption on $\Omega.$
It was first introduced in \cite{AS}, see also \cite{PS}.
\begin{assumption}\label{sheep}

$\Omega\subset \mathbb{C}^n$ is a connected smooth bounded pseudoconvex domain, such that
for any $z\in \partial\Omega, A^{\infty}(\Omega)\cap I(z)$ is dense in $A^2(\Omega).$

\end{assumption}


Recall the following simple

\begin{lemma}
Assumption \ref{sheep} is satisfied for bounded smooth strongly pseudoconvex domains or star-shaped
smooth pseudoconvex domains.
\end{lemma}

\begin{proof}

Notice that to verify Assumption \ref{sheep}, 
it suffices to check the following: for a given $z\in\partial\Omega,$
there exists a sequence $f_n\in A^{\infty}(\Omega)$ such that $f_n(z)=1$ and 
$\lim_{n\to \infty} \|f_n\|_{A^2(\Omega)}=0.$ 
Indeed, let $g\in A^{\infty}(\Omega).$
Then $g- g(z)f_n\in I(z)$ and $\lim_{n\to \infty} (g-g(z)f_n)=g$ in $A^2(\Omega).$
Thus, $A^{\infty}(\Omega)\cap I(z)$ is dense in $A^{\infty}(\Omega),$ and since $A^{\infty}(\Omega)$ is
dense in $A^2(\Omega) $ (Catlin \cite{Ca}), we are done.

Suppose that $\Omega$ is a smooth strongly pseudoconvex domain.
Let $z\in \partial\Omega.$
 It is well-known that $z$ is a peak point. Let $f\in A^{\infty}(\Omega)$ be
such that $f(z)=1, |f(w)|<1, w\in \overline{\Omega}\setminus {z}$. Then $\lim_{m\to\infty} \|f^m\|_2=0.$

 Now let $\Omega$ be a star shaped smooth domain. Without loss of generality, we may assume
 that $r\Omega\subset \Omega, 0\leq r\leq 1.$ Let $\theta\in \partial\Omega.$
 Let $f\in A^2(\Omega)$ be such that $\lim_{w\to \theta}(f(w))=\infty.$ 
Existence of such $f$ follows for example from [\cite{Ca2}, Lemma1, page 153]. 
 Then $f_r(z)=f(rz)\in A^{\infty}(\Omega)$
 and $\|f_r\|_2\leq r^{-2n}\|f\|_2,$ while $\lim_{r\to 1} f_r(\theta)=\infty.$ 
\end{proof}

We have another easy
\begin{lemma}

If $\Omega$ satisfies Assumption \ref{sheep}, then for any finite set $B\subset \partial\Omega,$
$A^{\infty}(\Omega)\cap I(B)$ is dense in $A^2(\Omega).$

\end{lemma}

\begin{proof}

Put $B=\lbrace z_i\rbrace_{1\leq i\leq m}.$
Let $\epsilon>0.$ 
Let $g\in A^{\infty}(\Omega).$ Let $\phi_i\in A^{\infty}(\Omega)$ be such that $\phi_i(z_j)=\delta_{ij}.$
Let $g_i\in A^{\infty}(\Omega)$
such that $g_i(z_i)=1, \|g_i\|<\epsilon$ (such $g_i$ exists by Assumption \ref{sheep}).
Then $g-\sum_i g(z_i)\phi_ig_i\in I(B)$ and 
$$\|\sum_i g(z_i)\phi_ig_i\|_2<\|g\|_{L^\infty(\Omega)}\sum_i\|\phi_i\|_{A^2(\Omega)}\epsilon.$$ 
Thus, $A^{\infty}(\Omega)\cap I(B)$ is dense in $A^{\infty}(\Omega),$ and since $A^{\infty}(\Omega)$ is
dense in $A^2(\Omega) ,$ we are done.

\end{proof}

 For $w\in \Omega,$ we will denote by $K_w\in A^2(\Omega)$ the reproducing
kernel of the Bergman space $A^2(\Omega).$ Thus $\langle g, K_w\rangle=g(w)$ for any $g\in A^2(\Omega).$ Also, denote by $k_w$ the normalized Bergman kernel $\frac{K_w}{\|K_w\|}.$

The following is the key result of this section.

\begin{lemma}\label{Nullst}

Suppose that domain $\Omega\subset \mathbb{C}^n$ satisfies Assumption \ref{sheep}.  
 Let $f=(f_1,\cdots, f_n):\overline{\Omega}\to \mathbb{C}^n$ 
be a an open holomorphic mapping.
If  $J_f$ is nonzero on $f^{-1}(0)\cap \overline{\Omega}$, then
$$(\sum f_iA^2(\Omega))^{\perp}=\sum_{w\in f^{-1}(0)}\mathbb{C}K_w.$$

\end{lemma}

\begin{proof}
Let us put $B=f^{-1}(0)\cap \Omega=\lbrace w_1,\cdots, w_m\rbrace$ and
$B'=f^{-1}(0)\cap \partial\Omega.$ It follows from  Corollary \ref{zero} that
$$\sum_i f_iA^{\infty}(\Omega)=I(f^{-1}(0))\cap A^{\infty}(\Omega).$$
Now we claim that $$(A^2(\Omega)\cap I(B))^{\perp}=\sum_{w\in f^{-1}(0)}\mathbb{C}K_w.$$
Indeed, it is clear that $K_w\perp (A^2(\Omega)\cap I(B))$ for all $w\in B.$ On the
other hand, since $K_{w}, w\in B$ are linearly independent and codimension
of $(A^2(\Omega)\cap I(B))$ in $A^2(\Omega)$ is at most $m=|B|$, we obtain the desired equality.

Thus it suffices to show that $\sum f_iA^2(\Omega)$ is dense in $A^2(\Omega)\cap I(B).$ 
It suffices to check that $I(f^{-1}(0))\cap A^{\infty}(\bar{\Omega})$ is dense in $A^2(\Omega)\cap I(B)$ by
Lemma 2.3.
 Let $f\in A^2(\Omega)\cap I(B),$ and let $f_n\in  A^{\infty}(\bar{\Omega})\cap I(B')$ be
such that $\lim_{n\to \infty} f_n=f$ in $A^2(\Omega).$ Let $g_i, i=1,\cdots, m$ be polynomials such that
$g_i(w_j)=\delta_{ij}, g_i(B')=0.$ Put $\phi_n=f_n-\sum_{i=1}^m f_n(w_i)g_i.$ Then $\phi_n(w_j)=0$
for all $j, n.$ Also, for  any $i, \lim_{n\to \infty} f_n(w_i)=0.$ Therefore, $\lim_{n\to \infty} \phi_n=f$
and $\phi_n\in I(f^{-1}(0))\cap  A^{\infty}(\bar{\Omega}).$ 
So, $I(f^{-1}(0))\cap A^{\infty}(\bar{\Omega})$ is dense in $A^2(\Omega)\cap I(B).$
\end{proof}

\section{Some geometry related to $\Omega, f$}

In the rest of the paper, we will fix once and for all a domain
$\Omega\subset \mathbb{C}^n$ satisfying Assumption \ref{sheep} and a
 holomorphic mapping \\$f=(f_1,\cdots,f_n):\overline{\Omega}\to \mathbb{C}^n$  in a  neighbourhood
of $\overline{\Omega}$ such that determinant of its Jacobian  $J_f$ is not identically 0.

Given a function $f:X\to Y$, we will denote by $X\times_fX$ the set \\$\lbrace (z, w)\in X\times X| f(z)=f(w)\rbrace.$

Let us introduce several notations related with $\Omega, f.$
Put $$Z=f^{-1}(f(V(J_f))), \Omega'=\Omega\setminus Z,$$ where $V(J_f)$ is the zero locus of $J_f$ in 
$\overline{\Omega}.$
We will also put \\$\Omega''=\Omega'\setminus f^{-1}(f(\partial \Omega)).$
Thus, $\Omega''\times_f\Omega''\subset \Omega'\times_f\Omega'$ are $n$-dimensional
complex manifolds. As usual $p_1, p_2:\Omega'\times_f\Omega'\to \Omega'$ denote the projections
on the first, second coordinate respectively. Clearly both $p_1, p_2$ are surjective finite-to-one
locally biholomorphic mappings.

Remark that $f:\Omega''\to f(\Omega'')$ is a proper locally biholomorphic mapping. Therefore
it is a covering. Also, $\Omega'$ is connected while $\Omega''$ might not be.

In this section we define a certain open subset $W_f$ of $\Omega'\times_f\Omega'$ which
will play a key role in the rest of the paper.

In this setting we have the following simple but useful

\begin{lemma}\label{proper}

Let $W$ be an open subset of $\Omega'\times_f\Omega'$ such that ${p_1}|_W:W\to \Omega'$ is a covering.
Then  $p_2|_W:W\to \Omega'$ is also a covering.
 In particular, $\partial(W)\subset \partial(\Omega')\times_f\partial(\Omega'),$
 and $p_1|_{\overline{W}}, p_2|_{\overline{W}}:\overline{W}\to \overline{\Omega}\setminus Z$ are coverings, where
 $\overline{W}$ denotes the closure of 
 $W$ in $(\overline{\Omega}\setminus Z)\times_f(\overline{\Omega}\setminus Z).$

\end{lemma}
\begin{proof}

Let $z\in \Omega'.$ Let $X\subset \overline{\Omega}$  be a closed set of measure 0
such that $X\cap f^{-1}(f(z))=\emptyset$ and $\Omega'\setminus X$ is simply connected. 
Then $p_1:W\setminus p_1^{-1}(X)\to \Omega'\setminus X$
 is an $m$-fold trivial covering for some $m$.
So there exist holomorphic 
embeddings $\rho_{i}:\Omega'\setminus X\to \Omega', 1\leq i\leq m$ such that for any  
$u\in \Omega'\setminus X$ we have 
$$p_{1}^{-1}(u)\cap W=\lbrace (u, \rho_i(u)), 1\leq i\leq m\rbrace.$$  Put 
$U=\Omega'\setminus f^{-1}(f(X)).$ Then $z\in U, f^{-1}(f(U))\cap \Omega=U$ and
$\Omega'\setminus U$ has measure 0. Since $\rho_i$  induces a bijection
on $f^{-1}(f(u))\cap \Omega$ for all $u\in U,$ it follows that
$\rho_i:U\to U$
is a bijection for all $1\leq i\leq m.$ 
Remark that the set of bijections $\lbrace\rho_i\rbrace_{1\leq i\leq m}$ is not closed under
taking compositions or inverses.

Therefore
$$p_2:p_2^{-1}(U)\cap W=\lbrace (\rho_i^{-1}(z), z), z\in U, 1\leq i\leq m\rbrace\to U$$ is an $m$-fold trivial covering. 
Since $U$ is a neighbourhood of $z$, we conclude that $p_2|_W:W\to \Omega'$ is a covering map.

Let $(a_n)=(z_n, w_n)\in W$ be a sequence in $W$ converging to the boundary $\partial(W).$ 
Since $p_1|_W, p_2|_W:W\to \Omega'$ are proper mappings as shows above, we get that both $(z_n), (w_n)$
converge to $\partial(\Omega').$ Therefore, $\partial(W)\subset \partial(\Omega')\times_f\partial(\Omega').$

Let $z'\in \partial(\Omega)\setminus Z.$ Let $Y\subset \Omega'$ be a simply connected 
open subset such that  $\overline{Y}$ contains a neighbourhood
of $z'$ in $\overline{\Omega}.$ Just as above, \\let $\rho_i:Y\to \Omega', 1\leq i\leq m$  be holomorphic
embeddings such that $$p_1^{-1}(Y)\cap W=\lbrace (y, \rho_i(y)), 1\leq i\leq m, y\in Y\rbrace.$$ 
Without loss of generality $\overline{\rho_i(Y)}\cap \overline{\rho_j(Y)}=\emptyset, i\neq j.$
Thus, $(z', \rho_i(z')), 1\leq i\leq m$ are distinct points in $p_1^{-1}(z')\cap \partial(W).$
By shrinking $Y$ further we may assume that each $\rho_i$ extends to a holomorphic
embedding from a neighbourhood of $\overline{Y}$ into a neighbourhood of $\overline{\Omega}.$
Now let $w\in \partial(\Omega)\setminus Z$ be such that $(z, w)\in \partial(W).$
Then, there is a sequence $(z_n, w_n)\in W$ converging to $(z', w).$ We may assume that
$z_n\in Y$ and $w_n=\rho_i(z_n)$ for a fixed $i.$ So $w=\rho_i(z').$ Therefore
$$\overline{W}\cap p_1^{-1}(\overline{Y})=\lbrace (y, \rho_i(y)), y\in \overline{Y}, 1\leq i\leq m\rbrace$$
Hence $p_1|_{\overline{W}}:\overline{W}\to \overline{\Omega}\setminus Z$ is an $m$-fold covering.

\end{proof}

Next we will define a certain open subset $W_f\subset \Omega'\times_f\Omega'$ which
will play a crucial role.

\begin{defin}\label{W}
Let $W_f\subset \Omega'\times_F\Omega'$ be the union of all connected components 
$W$ of $\Omega'\times_F\Omega'$ 
 such that the projection $p_1|_W:W\to \Omega'$ is a covering map.

\end{defin}

The following lemma summarizes  properties of $W_f$ that will be used later.

\begin{lemma}

$W_f$ is symmetric: if $(z, w)\in W_f$ then $(w, z)\in W_f.$
\\$W_f\times_{f}W_f=W_f$: if $(z, t)\in W_f$ and $(t, s)\in W_f$, then $(z, w)\in W_f.$

\end{lemma}

\begin{proof}

Since $p_1: \Omega'\times_F\Omega'\to \Omega'$ is a locally biholomorphic mapping, it follows that W
is the union of all $U\subset \Omega'\times_F\Omega'$ such that $p_1:U\to \Omega'$ is a covering.
  In fact, it is easy to see that 
 $W$ is a union of connected components of  $\Omega'\times_F\Omega'.$
In particular,  the diagonal $\lbrace (z, z), z\in \Omega'\rbrace$ is a connected component
of $W.$ It also follows from Lemma \ref{proper} that $W$ is symmetric: If $(z, w)\in W$ then $(w, z)\in W.$ 

Notice also that $W_f=W_{p_1}\times_{p_2}W$, where 
$$W_{p_1}\times_{p_2}W=\lbrace (z, w)\in W\times W| p_1(z)=p_2(w)\rbrace$$ 
denotes the pullback of $p_1, p_2:W\to \Omega'.$ Indeed, 
if $U\subset \Omega'\times_F\Omega'$
is a subset such that $p_1:U\to \Omega'$ is a covering, then so is
$U_{p_1}\times_{p_2}U\to \Omega'.$ Thus $U_{p_1}\times_{p_2}U\subset W.$
 
\end{proof}

Remark that if $f:\Omega\to f(\Omega)$ is a proper mapping, then $p_1:\Omega'\times_f\Omega'\to \Omega'$
is a covering, thus in this case $W_f=\Omega'\times_f\Omega'.$

\section{Commutants of $T_f$}

At first we show the following preliminary

\begin{lemma}\label{commute} 

Let $S:A^2(\Omega)\to A^2(\Omega)$ be a bounded
linear operator which commutes with $T_f=\lbrace T_{f_i}, i=1, \cdots, n\rbrace.$ Then there exists
a function $\Phi$ on $\Omega'\times_f\Omega'$ such that for any $g\in A^2(\Omega)$ we have
 $$S(g)(z)=\sum_{w\in f^{-1}(f(z))\cap \Omega}\Phi(z, w)g(w), z\in\Omega'.$$
Moreover, $\Phi$ is holomorphic on $\Omega''\times_f\Omega''.$
\end{lemma}

\begin{proof}
We claim that for any $z\in \Omega',$ we have $$S^{*}(K_z)\in \sum_{w\in f^{-1}(f(z))} \mathbb{C}K_{w}.$$
Indeed, given $g_i\in A^{\infty}(\Omega),$ then
$$\bigcap Ker(T_{g_i}^{*})=(\sum g_iA^2(\Omega))^{\perp}.$$ 
Applying this to $g_i=f_i-f_i(z),$ and  using \ref{Nullst} we get that
 $$\bigcap_i T^{*}_{f_i-f_i(z)}=\sum_{w\in f^{-1}(f(z))} \mathbb{C}K_{w}$$ and $S^{*}$ preserves this space.
 In particular we may write
 $$S^{*}(K_z)=\sum_{w\in f^{-1}(f(z))} \overline{\Phi(z, w)}K_{w}$$ for some $\Phi(z, w)\in \mathbb{C}.$
Thus for any $g\in A^2(\Omega)$, we have
$$\langle g, S^{*}(K_w)\rangle=\langle S(g), K_w\rangle=S(g)(w)=\sum_{w\in f^{-1}(f(z))} \Phi(z, w)g(w).$$
Recall that $\Omega''\to f(\Omega'')$ is a covering map. Thus, for any $z\in \Omega'',$
there exists an open neighbourhood $z\in U\subset \Omega''$
and holomorphic embeddings  
$\rho_1,\cdots ,\rho_m:U\to \Omega$  
such that $$f^{-1}(f(z))=\lbrace \rho_1(z),\cdots , \rho_m(z)\rbrace, z\in U.$$
Denote $\Phi(z, \rho_i(z))$ by $\phi_i(z).$ Thus,  
$$S(g)(z)=\sum_i \phi_i(z)g(\rho_i(z)), g\in A^2(\Omega), z\in U.$$

Fix $z\in U.$ Let us choose polynomials $g_1,\cdots, g_m\in \mathbb{C}[z_1,\cdots, z_n]$ 
 such that the matrix $A=g_i(\rho_j(z))$ is nondegenerate.
Thus, its inverse is a holomorphic matrix in a neighbourhood of $z.$ Therefore, 
$(\psi_i)_{1\leq i\leq m}=A^{-1}(S(g_i)_{1\leq i\leq m})$ is holomorphic.
So, $\Phi$ is holomorphic on $\Omega''\times_f\Omega''.$

\end{proof}

We have the following result, which is well-known when $\Omega$ is a unit disc in $\mathbb{C}$ and $f$ is a finite Blaschke
product. 

\begin{theorem}\label{key}

Suppose that a bounded linear operator $S:A^2(\Omega)\to A^2(\Omega)$ commutes with $T_f.$
Then there exists a holomorphic function $\Phi\in A(W_f)$ ($W_f$ as in Definition \ref{W}) 
such that for any $z\in \Omega', g\in A^2(\Omega)$ one has $S(g)(z)=\sum_{(z, w)\in W_f} \Phi(z, w)g(w).$

\end{theorem}

\begin{proof}

We know from Lemma \ref{commute} that there exists a function $\Phi$ on $\Omega'\times_f\Omega'$ such that 
$$S(g)(z)=\sum_{w\in f^{-1}(f(z))} \Phi(z, w)g(w), z\in \Omega', g\in A^2(\Omega).$$ Moreover, $\Phi$ is
holomorphic on $\Omega''\times_f \Omega'',$ where recall that $\Omega''=\Omega'\setminus f^{-1}(f(\partial\Omega))).$
Let us denote by $W'$ the support of $\Phi$ in $\Omega'\times_f\Omega'.$ We will prove that $p_1|_{W'}:W'\to \Omega'$
is a covering map.

Let $z\in \Omega'.$ 
Let $\Omega_1$ be a neighbourhood of $\overline{\Omega}$ such that $f$ is extends to a holomorphic mapping on it.
We will follow very closely Thompson's argument \cite{Th}.
Let $Y\subset \Omega'$ be a small neighbourhood of $z$, and let $\rho_1,\cdots, \rho_l:Y\to \Omega_1$
be holomorphic embeddings such that $$f(\rho_i(w))=w, f^{-1}(f(w))\cap \overline{\Omega}\subset \lbrace \rho_i(z)_{1\leq i\leq l}\rbrace.$$
Let $P_z\subset \lbrace 1,\cdots, l\rbrace$ be defined ass follows: $i\in P_z$ if there exists $w\in Y$
 so that $\rho_i(w)\in \Omega$ and  $\Phi(w, \rho_i(w))\neq 0.$ By making $Y$ smaller if necessary, we may
assume that $\overline{\rho_i(Y)}\cap \overline{\rho_j(Y)}=\emptyset$ for $i\neq j.$
We claim that for all $i\in P_z, \rho_i(Y)\subset \Omega.$ Indeed,
suppose that for some $i,$ $\rho_i(Y)$ is not a subset of $\Omega.$ 
Let $\epsilon>0$ be such that $$\epsilon<\frac{d(\rho_i(Y), \rho_j(Y))}{\sqrt{n}}, j\neq i.$$
For each $j\neq i$ let us pick $k$ such that $|z_k-w_k|>\epsilon$ for all $z\in \rho_i(Y), w\in \rho_j(Y).$ 
For $w\in Y,$ put $$h_i^w(z)=\prod_{j\neq i}(z_k-\rho_j(w)_k)\in \mathbb{C}[z_1,\cdots ,z_n].$$
Then $h_i^w(z)$ vanishes
  on $\rho_j(w), j\neq i$ and $h_i^w(\rho_i(w))\neq 0.$ It follows that
  $S(h^w_j(z))(w)=\langle h^w_j, S^{*}K_w\rangle$ 
 is a holomorphic
function on $U.$ Then the function $S(h^w_i(z))(w)=\Phi(w, \rho_i(w))h^w_i(\rho_i(w))$ is not identically 0, but
 vanishes on $\rho_i^{-1}(\Omega_1\setminus \Omega)$, which contains a nonempty
open subset by the assumption (recall that $\rho_i$ is an open mapping). Hence $S(h^w_i(z))(w)=0$ for all $w\in Y,$ a contradiction.

To summarize, we have holomorphic embeddings $\rho_i:Y\to \Omega_1, 1\leq i\leq l$ 
and a subset $P_z\subset \lbrace{1,\cdots , l\rbrace},$ such that 
$f(\rho_i(w))=F(w), w\in Y,$ and for any $i\in P_z, \rho_i(Y)\subset \Omega',$ there exists $w\in Y,$ so
that $\Phi(w, \rho_i(w))\neq 0.$ Moreover, $\Phi(w, \rho_j(w))=0$ for all $j\notin P_z.$
Thus, for any $w\in Y$ we have $\lbrace (w, \rho_i(w))_{i\in P_z}\rbrace=p_1^{-1}(w)\cap W'.$
Therefore $p_1|_{W'}:W'\to \Omega'$ is a covering. Hence, $W'$ is a union of connected components
of $W_f.$ Let us extend $\Phi$ to $W$ by 0 on $W\setminus W'.$ Then for any
$g\in A^2(\Omega), z\in \Omega'$ we have 
$$S(g)(z)=\sum_{(z, w)\in W_f} \Phi(z, w)g(w).$$
It can be shown that $\Phi$ is holomorphic exactly as in the end of the proof of Lemma \ref{commute}.

\end{proof}

The following statement follows immediately from the well-known
localisation property of the Bergman kernel [\cite{Oh}, Localisation Lemma, page 2], combined
with the transformation formula of the Bergman kernel function under a biholomorphic map.

\begin{prop}\label{kernel} Let $\Omega\subset \mathbb{C}^n$ be
a smooth bounded pseudoconvex domain. Let $z^1, z^2\in \partial \Omega$ and $z^1\in U_1, z^2\in U_2$
be open neighbourhoods, such that there exists a biholomorphic mapping
 $\rho:\overline{\Omega}\cap U_1\to \overline{\Omega}\cap U_2,$ so that $\rho(z^1)=z^2.$
Then $\|K_{w}\|=O(\|K_{\rho(w)}\|), w\in U_1\cap\Omega$ and $\lim_{w\to \partial\Omega} \|K_{w}\|= \infty.$
\end{prop}

Below we will use the following standard fact. We include its proof for a reader's convenience.
\footnote{Communicated to us by S. Sahutoglu}
\begin{lemma}\label{Sonm}

Let $\Omega\subset \mathbb{C}^n$ be a smooth bounded pseudoconvex domain.
Then $k_w\to 0$ weakly as $w\to \partial{\Omega}$

\end{lemma} 

\begin{proof}

Let $g\in A^2(\Omega).$ For $\epsilon>0$ let $g^{\epsilon}\in A^{\infty}(\overline{\Omega})$ be
such that $\|g-g^{\epsilon}\|_{A^2(\Omega)}<\epsilon.$ Then we have 
$$|\langle g, k_w\rangle|<\epsilon+\langle g^{\epsilon}, k_w\rangle\leq \epsilon+\|g^{\epsilon}\|_{L^{\infty}(\Omega)}/\|K_w\|_{A^2(\Omega)}$$
Therefore, $\lim \sup |\langle g, k_w\rangle|\leq \epsilon$ as $w\to \partial\Omega.$

\end{proof}

Before proceeding further, let us summarize various choices that we have made in relation to $f, W_f.$

\begin{prop}\label{notation}

\begin{itemize}

\item[(1)] There is an open subset $Y\subset \Omega'$ such that $\partial Y\cap \partial \Omega$ contains a nonempty subset of $\partial\Omega.$ There are holomorphic embeddings
$\rho_i:\bar{Y}\to \overline{\Omega}\setminus Z, 1\leq i\leq m$ such that 
$$p_1^{-1}(Y)\cap W_f=\lbrace (y, \rho_i(y)), y\in Y, 1\leq i\leq m\rbrace,$$
$$\rho_i(\partial(Y)\cap\partial\Omega)=\partial\Omega\cap \partial(\rho_i(Y)),  \rho_i(\bar{Y})\cap\rho_j(\bar{Y})=\emptyset, i\neq j.$$
 
\item[(2)]

There is an open subset $U\subset \Omega'$, such that $\Omega\setminus U$ has measure 0 and biholomorphic mappings $\rho_i:U\to U, 1\leq i\leq m$
such that $$p_1^{-1}(U)\cap W_f=\lbrace (z, \rho_i(z)), z\in U, 1\leq i\leq m\rbrace.$$

\end{itemize}

\end{prop}

\begin{proof}

Let $Y\subset \Omega\setminus Z$ be a an open subset such that $\bar{Y}$ is simply connected and $\partial(Y)\cap \partial\Omega$ 
contains an open subset of $\partial\Omega.$ Thus $p_1:p_1^{-1}(\bar{Y})\cap \overline{W_f}\to \bar{Y}$ is a trivial covering.
Therefore there exist holomorphic mappings $\rho_i:\bar{Y}\to \overline{\Omega}\setminus Z, 1\leq i\leq m$ such that 
$$p_1^{-1}(Y)\cap W_f=\lbrace (y, \rho_i(y)), y\in Y, 1\leq i\leq m\rbrace.$$ Recall that
 $\partial W_f\subset \partial\Omega\times_f\partial\Omega.$ Therefore, 
 $\rho_i(\bar{Y}\cap \partial\Omega)=\rho_i(\bar{Y})\cap \partial\Omega.$ By shrinking Y further, we get that $\rho_i(\bar{Y})\cap\rho_j(\bar{Y})=\emptyset, i\neq j.$

Part (2) follows directly from the proof of Lemma \ref{W}.

\end{proof}

We have the following

\begin{theorem}\label{compact}

Let $S:A^2(\Omega)\to A^2(\Omega)$ be a compact operator such that it commutes with $T_f.$ 
Then $S=0.$
\end{theorem}

\begin{proof}

We will use notations from Proposition \ref{notation}. It follows from Theorem \ref{key} and its proof that there are holomorphic functions $\phi_i\in A(Y)$ such that 
$$S(g(w))=\sum_i \phi(w)g(\rho_i(w)), w\in Y$$

Next we will look at the two variable Berezin transform of $S.$ 
 Since $S$ is a compact operator and since by Lemma \ref{Sonm} $\frac{K_w}{||K_w||}\to 0$ weakly as $w\to \partial{\Omega}$, we have
$$\lim_{w_1, w_2\to \partial\Omega}\frac{\langle S(K_{w_1}), K_{w_2}\rangle}{\|K_{w_1}\|\|K_{w_2}\|}=0.$$

Recall $\epsilon>0,$ and functions $h_i^w(z)=\prod_{j\neq i}h_{ij}(z, w),$
from the proof of Theorem \ref{key}: here $h_{ij}(z, w)=(z_k-\rho_j(w)_k)$ is linear in $z$ such that
$$|h_{ij}(z, w)|\geq \epsilon, z\in \rho_i(Y), w\in Y, i\neq j.$$
Since $\Omega$ is bounded, there exists $M>0$ such that $\|h_i^w(z)\|<M$ for all $i, z\in \Omega, w\in Y.$ 
Thus, for all $w\in Y.$
$$|\langle S(h_i^w), K_w\rangle|\leq M\|S\|\|K_w\|$$ 
Then, $$\langle S(h_i^w), K_w\rangle=\sum_j\phi_j(w)h_i^w(\rho_i(w))=\phi_i(w)\prod_{j\neq i}h_{ji}(\rho_j(w), \rho_i(w)).$$
By our assumption $$\prod_{j\neq i}|h_{ji}(\rho_j(w), \rho_i(w))|\geq \epsilon^{m-1}.$$
This implies that there is $N$ such that $|K_{\rho_i(w)}(\rho_j(w))|<N$ for all $i\neq j, w\in Y.$
Thus, there exists $L>0,$ such that $\phi_i(w)\leq L||K_w||$ for all $i, w\in Y.$

We have $$\langle S(K_{\rho_i(w)}), K_w\rangle=\sum_j\phi_j(w)K_{\rho_i(w)}(\rho_j(w)).$$
 So, for $i\neq j$ we have 
 $$\lim_{w\to \partial\Omega\cap \partial U}\frac{\phi_j(w)K_{\rho_i(w)}(\rho_j(w))}{\|K_w\|\|K_{\rho_i(w)}\|}=0.$$
Therefore, $$\lim_{w\to \partial\Omega\cap \partial Y}\frac{\phi_i(w)\|K_{\rho_i(w)}\|}{\|K_w\|}=0,$$
which by Proposition \ref{kernel} implies that $\lim_{w\to \partial\Omega\cap \partial Y}\phi_i(w)=0$ for all $i.$
This implies that $\psi_i=0$ for all $i$ by the Boundary uniqueness theorem [\cite{Ci}, page 289].

\end{proof}

\begin{lemma}\label{peloso}

Suppose that
$H_{\bar{z_i}}$ (the Hankel operator with symbol $\overline{z_i}$) is compact for all $i.$
Let $G=\lbrace g_1,\cdots, g_m\rbrace$  be an $m$-tuple of bounded
holomorphic functions on $\Omega$  such that the commutant of $T_G=\lbrace T_{g_i}, 1\leq i\leq m\rbrace$
 contains no nonzero compact operators.
If an operator $S$ in the Toeplitz algebra of $\Omega$ commutes
with $T_G,$ then $S$ is a multiplication operator by a
 bounded holomorphic function on $\Omega.$

\end{lemma}

\begin{proof}
Recall that for any $g\in L^{\infty}(\Omega)$, we have  
$[T_{z_i}, T_g]=H^{*}_{\bar{z_i}}H_g.$ This equality combined with compactness of
$H_{\bar{z_i}}$ implies that for any element $S$ of the Toeplitz algebra of $\Omega,$
operators $[T_{z_i}, S], 1\leq i\leq n$ are compact. 
If in addition $S$ commutes with $T_G,$ then $[T_{z_i}, S], 1\leq i\leq n$ are compact operators in the
commutant of $T_G.$ Thus  $[T_{z_i}, S]=0, 1\leq i\leq n.$
Now by \cite{SSU} $S=T_h$ for some $h\in H^{\infty}(\Omega)$. 

\end{proof}

It is well-known that smooth strongly pseudoconvex domains satisfy the assumption in Corollary \ref{peloso}
(follows immediately from [\cite{Pe}, Theorem 1.2]). Hence as a consequence of Lemma \ref{peloso} and 
Theorem \ref{compact} we obtain the following

\begin{theorem}\label{strong}
Let $\Omega\subset \mathbb{C}^n$ be a bounded smooth strongly pseudoconvex domain.
Let $f=(f_1,\cdots, f_n):\overline{\Omega}\to \mathbb{C}^n$ be a holomorphic
mapping on a neighbourhood of $\overline{\Omega}$ with a nontrivial Jacobian determinant. 
If $S$ is an element of the Toeplitz algebra of $\Omega$ which commutes with
 $T_{f_i}, i=1,\cdots, n,$ then  $S$ is a multiplication operator by
a bounded holomorphic function on $\Omega$.

\end{theorem}

Recall that in general, given an $(n-1)$-tuple of holomorphic functions $f_1,\cdots, f_{n-1}$ on $\Omega,$
 commutants of $T_{f_i},1\leq i\leq n-1$ will contain nontrivial compact operators [\cite{Le}, Proposition 2.4].
However, it is possible that for a specific $f\in A^{\infty}(\Omega),$ no nontrivial compact operator commutes with
$T_f$ [\cite{Le}, Theorem 1.1].







\section{Convolution algebras}

Let $f:X\to Y$ be a finite-to-one local homeomorphism
of topological spaces. Recall the standard notation
$$X\times_fX=\lbrace (z, w)\in X\times X, f(z)=f(w)\rbrace.$$ We have two projections
$$p_1, p_2: X\times_fX\to X, p_1(z, w)=z, p_2(z, w)=w.$$ 
Also recall that for a subset $Z\subset X\times_fX$ we have 
$$Z_{p_1}\times_{p_2}Z=\lbrace(z, w)\in X\times_fX|\exists t\in X s.t. (z, t)\in Z, (t, w)\in Z\rbrace.$$

Let $W$ be a symmetric subset of $X\times_fX$: if $(x_1, x_2)\in W$ then $(x_2, x_1)\in W$ such that $p_1|_{W}:W\to X$ is a covering
and $W_{p_1}\times_{p_2}W=W.$
Recall that in this setting $(\mathbb{C}[W]$
($\mathbb{C}$-valued continuous functions on $W$) is an associative algebra under  the 
convolution product $\star$:
$$\phi\star \psi(z, w)=\sum_{(z, t), (t, w)\in W}\phi(z, t)\psi(t, w), \phi, \psi\in \mathbb{C}[W].$$
Given $g\in \mathbb{C}[W],$ one defines the corresponding
weighted composition operator $S_g:\mathbb{C}[X]\to \mathbb{C}[X]$ as follows
$$S_g(\phi)(x)=\sum_{(x, w)\in W} g(x, w)\phi(w), \phi\in \mathbb{C}[X], x\in X.$$ This way $\mathbb{C}[X]$ becomes
a left $(\mathbb{C}[W], \star)$-module. It is straightforward to check that $S_g$ commutes with $T_f,$
where $T_f:\mathbb{C}[X]\to \mathbb{C}[X]$ is the multiplication operator
by $f.$

If in addition $X, Y, W$ are complex manifolds and $f$ is locally biholomorphic mapping, then $A(W)$ (the space of all holomorphic
functions on $W$) is a subalgebra of $(\mathbb{C}[W], \star)$. 
\begin{defin}\label{convolution}
Let $f:X\to Y, W\subset X\times_fX$ be as above.
We will denote by $\mathcal{A}(W)$ the algebra of all locally constant functions on $W$ under 
the convolution product.
If $f:X\to Y$ is a finite covering, then we will denote $\mathcal{A}(X\times_fX)$ by $\mathcal{A}(X, f).$
\end{defin}

If $f:X\to Y$ is a finite covering, and $X, Y$ are path connected, locally simply connected spaces,
 then $\mathcal{A}(X, f)$ 
can be naturally identified with the Hecke algebra 
of all bi -$\pi_1(X)$-invariant $\mathbb{C}$-valued functions on $\pi_1(Y)$ under the convolution product.
In particular, if $f:X\to Y$ is a normal covering, then $\mathcal{A}(X, f)$  is isomorphic to the group
algebra $\mathbb{C}[\pi_1(Y)/f_{*}\pi_1(X)].$

Let $Y'\subset Y.$ Then $f:X'=f^{-1}(Y')\to Y'$ is a covering map, and we have an algebra
homomorphism $\mathcal{A}(X, f)\to \mathcal{A}(X', f')$ given by the restriction of elements of $\mathcal{A}(X, f)$ on $X'\times_fX'.$

Let $f:M\to N$ be a finite covering map of connected real manifolds with boundaries. Then we get restrictions
of $f$ which are again coverings 
$f:M\setminus \partial(M)\to N\setminus \partial(N), f:\partial(M)\to \partial(N).$

In this setting we have the following simple
\begin{lemma}\label{abelian}
Suppose that $\partial(M)$(hence $\partial(N)$) is connected and $\pi_1(\partial(N))$ is Abelian. Then
$\mathcal{A}(M, f)=\mathcal{A}(M\setminus\partial(M), f)$ is commutative.

\end{lemma}

\begin{proof}

We have $\partial(M\times_f M)=\partial(M)\times_f \partial(M).$ Let $X'$ be
a connected component of $M\times_f M.$ Then $p_1:X'\to M$ is a covering map, hence
$\partial(X')$ is a nonempty component of $\partial(M)\times_f \partial(M).$ Hence,
if $\phi \in \mathcal{A}(M, f)$ is such that $\phi_|{X'}\neq 0$ then the image
of $\phi$ in $\mathcal{A}(\partial(N), f)$ is nonzero on $\partial(X').$ So, $\mathcal{A}(M, f)$ embeds
into $\mathcal{A}(\partial(M), f).$ Since $X'\setminus\partial(X')=X'\setminus (\partial(M)\times_f \partial(M))$
is connected, we obtain that $\mathcal{A}(M, f)=\mathcal{A}(M\setminus \partial(M), f).$
Since $\pi_1(\partial(N))$ is Abelian, $\partial M\to \partial N$ is a normal covering.
Therefore $\mathcal{A}(\partial(M), f)=\mathbb{C}[\pi_1\partial(N)/\pi_1\partial(M)].$ Hence
$\mathcal{A}(\partial(M), f)$ is commutative. This implies that $\mathcal{A}(M\setminus \partial(M), f)$ is also commutative.

\end{proof}






\section{Commutants of  $\lbrace T_f, T_f^{*}\rbrace$}

The following assumption on the mapping $f$ will play a key role.

\begin{assumption}\label{dense}

Assume that $Z=f^{-1}(f(V(J_f)))$ is not dense in the Zariski topology of $\Omega:$
There exists a nonzero  $g\in A^{\infty}(\Omega)$ such that $g(Z)=0.$
\end{assumption}

This assumption is satisfied if $f$ is a rational mapping,
 if $n=1$, or
$f:\Omega\to f(\Omega)$ is a proper mapping \cite{Ru}.

The following is the main result of the paper.

\begin{theorem}\label{reducing} Assume that Assumption \ref{sheep} holds for $\Omega.$
Then under the notations of Theorem \ref{key}, 
the algebra of  commutants of $\lbrace T_f, T_f^{*}\rbrace$ is isomorphic to a subalgebra of
$\mathcal{A}(W_f),$-the algebra of  locally constant
functions on $W$ under convolution (Definition \ref{convolution}). If in addition mapping $f$ satisfies Assumption \ref{dense}, then
these algebras are isomorphic.

\end{theorem}

\begin{proof}
Recall that $p_1|W_f:W_f\to \Omega'$ is a covering. From now on we will denote $p_1|W_f$ by
$p_1$ for simplicity. Similarly, $p_2|_{W_f}$ will be abbreviated to $p_2.$
We will define an algebra homomorphism 
$$\iota: A(W_f)\to Hom_{\mathbb{C}}(A(\Omega'), A(\Omega'))$$ as follows.
Let $c\in A(W_f), \phi\in A(\Omega').$ We will define a holomorphic
function $\iota_c(\phi)\in A(\Omega')$ in the following way. 
We put 
$$\iota_c(\phi)(z)=\sum_{(z, w)\in W}c(z, w)\frac{J_f(z)}{J_f(w)}\phi(w), z\in \Omega'.$$
Clearly $\iota_c(\phi)\in A(\Omega').$  
It is straightforward to check that $\iota$ is an algebra homomorphism.
To define $\iota_c(\phi)$ more explicitly we will use notations from Proposition $\ref{notation}$
Recall that by the chain rule  $J_{\rho_i}(z)=\frac{J_F(z)}{J_F(\rho_i(z))}.$
Therefore  
$$\iota_c(\phi)(z)=\sum_ic(z, \rho_{i}(z))J_{\rho_i}(z)\phi(\rho_{i}(z)), z\in \Omega'.$$

In what follows given $g\in A(\Omega'), z\in \Omega'$, by $J_{\rho}g(\rho(z))$ we will
denote the column vector   $(J_{\rho_i}(z)g(\rho_i(z)))_{1\leq i\leq m}$ in $\mathbb{C}^m.$
Now we follow very closely Guo-Huang [\cite{GH}, the proof of Proposition 3.4].

\begin{lemma}
Suppose that $S:A^2(\Omega)\to A^2(\Omega)$ commutes
with $T_f.$ Let $U\subset \Omega'$ be as above. Then there exists a holomorphic mapping
$\Phi:U\to gl_m(\mathbb{C})$ such that $J_{\rho}S(g)(\rho(z))=\Phi(z) J_{\rho}g(\rho)(z).$

\end{lemma}

\begin{proof}

Using  Theorem \ref{key}, there exists $c\in A(W)$ such that
$$S(g)(z)=\sum_iJ_{\rho_i}(z)c(z, \rho_i(z))g(\rho_i(z))=\sum_{(z, w)\in W}c(z, w)\frac{J_f(z)}{J_f(w)}g(w).$$
Then the $i$-th coordinate of the vector$J_{\rho}S(g)(\rho(z))$ is 
$$\frac{J_f(z)}{J_f(w)}\sum_{\tau\in p_1^{-1}(w)} \frac{J_f(w)}{J_f(\tau)}c(w, \tau)g(\tau), w=\rho_i(z).$$
Let us put $\Phi(z)_{jk}=c(\rho_j(z), \rho_k(z)).$ 
Now it follows easily that $$J_{\rho}S(g)(\rho(z))=\Phi(z) J_{\rho}g(\rho)(z).$$


\end{proof}

 Now assume that both $S, S^*$ 
commute with $T_f.$ Then by the above lemma there exist holomorphic mappings $\Phi, \Psi:U\to gl_m(\mathbb{C})$
such that
 $$J_{\rho}S(g)(\rho(z))=\Phi(z) J_{\rho}g(\rho)(z), J_{\rho}S^{*}(g)(\rho(z))=\Psi(z) J_{\rho}g(\rho)(z).$$
 Let $\lambda, \mu\in \Omega.$
 Given two polynomials $P, Q\in \mathbb{C}[x_1,\cdots, x_n]$ we have 
 $$\langle P(T_f)S(K_\lambda), Q(T_f)K_{\mu}\rangle=\langle P(T_f)(K_\lambda), Q(T_f)S^{*}(K_{\mu})\rangle.$$
So $$\int_{U}P\bar Q(f)(z)S(K_{\lambda})\bar{K_{\mu}}dV(z)=\int_{U}P\bar Q(F)(z)(K_{\lambda})\overline{S^{*}(K_{\mu})}dV(z)$$

Using the Stone-Weierstrass approximation, we see that for any $g\in C(\overline{F(\Omega}))$ one has
$$\int_{U}g(F(z))S(K_{\lambda})\bar{K_{\mu}}d_zV=\int_{U}g(F(z))(K_{\lambda})\overline{S^{*}(K_{\mu})}d_zV.$$
Thus the same equality holds for any $g\in L^{\infty}(\overline{F(\Omega)}).$ This implies using
change of variables that for all $ z\in U$
$$\sum_j|J_{\rho_{j}}(z)|^2S(K_{\lambda})(\rho_{j}(z))\overline{K_{\mu}(\rho_{j}(z))}=\sum_j |J_{\rho_{j}}(z))|^2K_{\lambda}(\rho_{j}(z))\overline{S^{*}(K_{\mu})(\rho_{j}(z)))},$$
the latter equality can be rewritten as 
$$\langle \Phi(z)J_{\rho}(z)K_{\lambda}(\rho(z)),  J_{\rho}(z)K_{\mu}(\rho(z))\rangle=\langle J_{\rho}(z)K_{\lambda}(\rho(z)), \Psi(z)J_{\rho}(z)K_{\mu}(\rho(z))\rangle,$$
where inner product is the standard one in $\mathbb{C}^m.$
Next we will use the following simple 

\begin{lemma}
For any $z\in \Omega'$ vectors $\lbrace J_{\rho}(z)K_{\lambda}(\rho(z))\rbrace_{\lambda\in \Omega}$ span $\mathbb{C}^m.$

\end{lemma}

\begin{proof}
Let vector $a=(a_i)_{i=1}^m\in \mathbb{C}^m$ be perpendicular to 
$\lbrace J_{\rho}(z)K_{\lambda}(\rho(z))\rbrace_{\lambda\in \Omega}.$
Thus for all $\lambda\in \Omega$
$$0=\sum_{i=1}^ma_i\overline{J_{\rho_i}(z)K_{\lambda}(\rho_i(z))}=\sum_{i=1}^ma_i\overline{J_{\rho_i}(z)}K_{\rho_i(z)}(\lambda).$$
Since $J_{\rho_i}(z)\neq 0$ and $K_{\rho_i(z)}, 1\leq i\leq m$ are linearly independent, it follows that
$a=0.$

\end{proof}
Now it follows from the above Lemma that
 $\Psi(z)$ is the adjoint of $\Phi(z).$ Since $\Phi, \Psi$ are holomorphic, it follows that  $\Phi, \Psi$ are locally constant functions on $U.$

Thus, we conclude that if $S:A^2(\Omega)\to A^2(\Omega)$ is a bounded linear operator such that
$S, S^{*}$ commute with $T_f$, then there exists a locally constant function $c$ on $W_f,$ such that
$S=\iota_c.$ This implies that the algebra of commutants of $\lbrace T_f, T_f^{*}\rbrace$ is
isomorphic to a subalgebra of $\mathcal{A}(W_f).$

Now let us assume that Assumption 2 is satisfied. Therefore, by Bell's result $A^2(\Omega')=A^2(\Omega)$ [\cite{Be}, Removable singularity theorem].
Next, suppose that $c\in H^{\infty}(W_f)$ is bounded holomorphic function on $W$ and $\phi\in A^2(\Omega).$ Then we claim
that $\iota_c(\phi)\in A^2(\Omega).$ Indeed, it follows from the change of variables that for all $1\leq i\leq m.$
$$\|c(z, \rho_i(z))J_{\rho_i}(z)\phi(\rho_i(z))\|_{L^2(U)}\leq \|c\|_{L^{\infty}(W)}\|\phi\|_{L^2(\Omega')}.$$
Therefore, $$\|\iota_c(\phi)\|_{A^2(\Omega')}\leq m \|c\|_{L^{\infty}(W)}\|\phi\|_{L^2(\Omega')}.$$
\noindent Hence $\iota_c(\phi)\in A^2(\Omega).$

Let $c\in \mathcal{A}(W).$
Put $c^{*}(z, w)=\overline{c(w, z)}, (z, w)\in W.$

Let $\phi, \psi\in A^2(\Omega).$ 
We have
$$\langle \iota_c(\phi), \psi\rangle_{A^2(\Omega)}=\sum_j \int_{U} c(z, \rho_{j}(z))J_{\rho_{j}}(z)\phi(\rho_{j}(z))\overline{\psi(z)}dV(z)=$$
$$\sum_j\int_{\rho_{j}(U)}\phi(w)c({\rho_{j}}^{-1}(w), w)\overline{J_{{\rho_{j}}^{-1}}}(w)\overline{\psi({\rho_{j}}^{-1}(w))}dV(w),$$

the latter is $\langle \phi, \iota_{c^*}(\psi)\rangle_{A^2(\Omega)}.$
Thus, we have shown that for any \\$c\in \mathcal{A}(W), \iota_c: A^2(\Omega)\to A^2(\Omega)$ is a bounded linear
operator commuting with $T_f.$ Moreover $(\iota_c)^{*}=\iota_{c^{*}}.$ This concludes the
proof.

\end{proof}



As a consequence, we reprove the following theorem of Douglas, Putinar and Wang  [\cite{DPW}, Theorem 2.3].

\begin{theorem}

Let $f\in A^{\infty}(D)$ be a finite Blaschke product on the unit
 disc $D.$ Then the algebra of commutants of $\lbrace T_f, T_f^{*}\rbrace$
is isomorphic to $\mathbb{C}\underbrace{\oplus\cdots\oplus}_q \mathbb{C},$ 
where $q$ equals  the number of irreducible
components of $D'\times_{f} D'.$ 

\end{theorem}

 \begin{proof}
 It follows from Definition \ref{convolution} that $\dim_{\mathbb{C}} \mathcal{A}(D', f)=q.$
  The algebra of commutants
 of $\lbrace T_{f}, T_{f}^{*}\rbrace$  is isomorphic to $\mathcal{A}(D', f)$ by
 Theorem \ref{reducing}.
But  $\mathcal{A}(D', f)$
 is isomorphic to a subalgebra of $\mathcal{A}(\partial(D), f)$ by Lemma \ref{abelian}, which is
 commutative since $\pi_1(\partial D)=\mathbb{Z}$ is Abelian. Thus, the algebra of commutants of $\lbrace T_f, T_f^{*}\rbrace$
   is a $q$-dimensional commutative Von Neumann algebra, hence it is isomorphic to $\mathbb{C}\underbrace{\oplus\cdots\oplus}_q \mathbb{C}.$

 \end{proof}

\acknowledgement{ I am very grateful to  Z.~Cuckovic for introducing me to this area, in particular
for bringing the results of \cite{DSZ} to my attention. Thanks are also due to T.~Le and especially to S.~Sahutoglu
for countless helpful discussions.}


\begin{thebibliography}{}

\bibitem[AS]{AS}

O~. Agrawal, N.~Salinas, {\em Sharp kernels and canonical subspaces}, Amer. J. Math. 109 (1987), 23--48.

\bibitem[ACR]{ACR}
S.~Axler, Z.~Cuckovic, N.~Rao, {\em Commutants of analytic Toeplitz operators on the Bergman space},
 Proc. Amer. Math. Soc. 128 (2000), no. 7, 1951--1953.


\bibitem[Be]{Be}

S.~Bell, {\em The Bergman kernel function and proper holomorphic mappings }, Trans. Amer. Math. Soc. 270 (1982), no. 2, 685--691.



\bibitem[Ca]{Ca}
D.~ Catlin, {\em Boundary behavior of holomorphic functions on pseudoconvex domains},  J. Differential Geom. 15 (1980), no. 4, 605--625.

\bibitem [Ca2]{Ca2}
D.~Catlin, {\em Necessary conditions for subellipticity of the $\bar{\partial}$-Neumann problem},
Ann. of Math. (2) 117 (1983), no. 1, 147--171.


\bibitem [Ci]{Ci}
E.~Cirka, {\em Complex Analytic Sets}, (1989) Kluwer Academic Publishers. 


\bibitem [Co]{Co}

C.~Cowen, {\em The commutant of an analytic Toeplitz operator}, Trans. Amer. Math. Soc. 239 (1978), 1--31.

\bibitem [Do]{Do}

R.~Douglas, {\em Banach algebra techniques in operator theory},  GTM, Springer-Verlag (1998).

\bibitem[DSZ]{DSZ}
R.~Douglas, S.~Sun, D.~Zheng, {\em Multiplication operators on the Bergman space via analytic continuation}, 
 Adv. Math. 226 (2011), no. 1, 541--583.

\bibitem[DPW]{DPW}
R.~Douglas, M.~Putinar, K.~Wang, {\em Reducing subspaces for analytic multipliers of the Bergman space}, 
 J. Funct. Anal. 263 (2012), no. 6, 1744--1765.
\bibitem[GH]{GH}
K.~Guo, H.~Huang, {\em  On multiplication operators on the Bergman space: similarity, unitary equivalence and reducing subspaces},
 J. Operator Theory 65 (2011), no. 2, 355--378.
\bibitem[GH2]{GH2}
K.~Guo, H.~Huang, {\em Multiplication operators defined by covering maps on the Bergman space: the connection between operator theory and von Neumann algebras},  J. Funct. Anal. 260 (2011), no. 4, 1219--1255. 
 J. Operator Theory 65 (2011), no. 2, 355--378.
\bibitem[HS]{HS}
M.~Hakim, N.~Sibony, {\em  Spectre de $A(\overline{\Omega}$) pour des domaines bornés faiblement pseudoconvexes réguliers}, 
 J. Funct. Anal. 37 (1980), no. 2, 127--135.
\bibitem[Le]{Le}
T.~Le, {\em The commutants of certain Toeplitz operators on weighted Bergman spaces },  
J. Math. Anal. Appl. 348 (2008), no. 1, 1--11.

\bibitem[PS]{PS}
M.~Putinar, N.~Salinas,   {\em Analytic transversality and nullstellensatz in Bergman space}, (1993)
367--381, Contemp. Math., 137, AMS.
\bibitem[Ov]{Ov}
N. Ovrelid, {\em Generators of Maximal Ideals of $A(\overline{D})$},  Pacific J. Math. 39 (1971), 219--223. 

\bibitem[Oh]{Oh}
T.Ohsawa, {\em Boundary behaviour of the Bergman kernel function on pseudoconvex domains},  
Publ. Res. Inst. Math. Sci. 20 (1984), no. 5, 897--902.

\bibitem[Pe]{Pe}
M.~Peloso, {\em Hankel operators on weighted Bergman spaces on strongly pseudoconvex domains},
 Illinois J. Math. 38 (1994), no. 2, 223--249.


\bibitem[Ru]{Ru}
W.~Rudin, {\em  Function theory on the unit ball of $\mathbb{C}^n$}, Springer-Verlag, 1980.


\bibitem[SSU]{SSU}
N.~ Salinas, A.~ Sheu, H.~ Upmeier,
{\em Toeplitz operators on pseudoconvex domains and foliation $C^{ *}$-algebras}, Ann. Math. Vol. 130, No. 3 (1989),  531--565.


\bibitem[Th]{Th}

J.~Thompson, {\em The commutant of a class of analytic Toeplitz operators}, Amer. J. Math. 99 (1977),
522--529.





\end{thebibliography}
\end{document}